\documentclass{article}
\usepackage[a4paper, left=80pt, right=80pt, headsep=16pt, footskip=28pt]{geometry}
\usepackage{authblk}

\usepackage{color,xcolor,graphicx}
\definecolor{darkgreen}{rgb}{0,0.4,0}

\usepackage{hyperref}
\hypersetup{
	breaklinks=true,
	colorlinks=true,
	linkcolor=darkgreen,
	citecolor=darkgreen,
	urlcolor=black,
}

\usepackage{float}
\usepackage{amsfonts, amssymb, amsmath, amsthm, amstext, mathtools}
\usepackage{thmtools}

\usepackage[noabbrev,capitalise,nosort,nameinlink]{cleveref}

\usepackage{bbm}

\usepackage{enumerate, enumitem, moreenum}
\usepackage{natbib}
\setcitestyle{authoryear, aysep={ }, comma}

\numberwithin{equation}{section}
\numberwithin{figure}{section}
\numberwithin{table}{section}

\usepackage{parskip}

\newcommand{\MM}[1]{\bgroup\color{olive}MM:~#1\egroup}
\newcommand{\CF}[1]{\bgroup\color{cyan}CF:~#1\egroup}

\theoremstyle{plain}
\newtheorem{theorem}{Theorem}[section]

\newtheorem{lemma}[theorem]{Lemma}
\newtheorem{corollary}[theorem]{Corollary}
\theoremstyle{definition}
\newtheorem{definition}[theorem]{Definition}

\newtheorem{remark}[theorem]{Remark}

\renewcommand{\epsilon}{\varepsilon}

\newcommand{\norm}[1]{\lVert #1 \rVert}

\newcommand*{\quark}{\setbox0\hbox{$x$}\hbox to\wd0{\hss$\cdot$\hss}}

\newcommand*{\R}{\mathbb{R}}
\newcommand*{\Rnn}{\mathbb{R}_{\geq 0}}
\newcommand*{\Rp}{\mathbb{R}_{>0}}
\newcommand*{\N}{\mathbb{N}}

\newcommand*{\E}{\mathbb{E}}
\renewcommand*{\P}{\mathbb{P}}

\newcommand{\sigalg}{\mathcal{F}}

\newcommand{\X}{\mathcal{X}}

\newcommand{\cX}{\mathcal{X}}
\newcommand{\cY}{\mathcal{Y}}

\newcommand{\bz}{\mathbf{z}}

\newcommand{\bZ}{\mathbf{Z}}

\newcommand{\dd}{\mathrm{d}}

\DeclareMathOperator*{\esssup}{ess\,sup}

\title{
Fuk--Nagaev inequality in smooth Banach spaces:\\
Optimum bounds for distributions of heavy-tailed
martingales
}

\author[1]{Mattes Mollenhauer\thanks{mattes.mollenhauer@merantix-momentum.com}}
\author[2]{Christian Fiedler\thanks{christian.fiedler@tum.de}}
\affil[1]{Merantix Momentum}
\affil[2]{Technical University of Munich and Munich Center for Machine Learning (MCML)}

\date{\today}

\begin{document}

\maketitle

\begin{abstract}
    We derive a Fuk--Nagaev inequality
    for the maxima of norms of martingale sequences
    in smooth Banach spaces which allow
    for a finite number of higher conditional moments.
    The bound is obtained by combining an optimization approach 
    for a Chernoff bound
    due to \citet{rio2017constants} with a classical bound for
    moment generating functions of smooth Banach space norms by \citet{Pinelis94}.
    Our result improves comparable infinite-dimensional bounds in the literature by removing unnecessary 
    centering terms and giving precise constants.
    As an application, we propose a McDiarmid-type bound
    for vector-valued functions which allow for
    a uniform bound on their conditional higher moments.
\end{abstract}

\section{Introduction}

The family of \emph{Fuk--Nagaev inequalities}
\citep{FukNagaev1971, Fuk1973, Nagaev1979} provides tail bounds
for independent sums and martingales under finite higher moment assumptions, thereby generalizing classical 
results for bounded, subgaussian, and subexponential 
random variables by Hoeffing, Azuma, Bennett, and Bernstein
(see e.g.\ \citealp{boucheron2013})
to distributions with heavier tails.

In particular, for centered independent real-valued random variables
$X_1, \dots, X_n$ satisfying the moment conditions
$\sum_{i=1}^n \E|X_i|^2 \leq \sigma^2$ and
$\sum_{i=1}^n \E|X_i|^q \leq C_q^q$ for some $q > 2$,
the sum $S_n = X_1 + \dots + X_n$ allows for the tail bound
\begin{equation*}
    \P \left[ \left| S_n \right|
    > t
    \right]
    \leq
    \alpha_q
    \frac{ C_q^q}{t^q}
    +
    \beta_q
    \exp \left(- \gamma_q
    \frac{t^2}{\sigma^2}
    \right)
\end{equation*} for all $t >0$ and constants
$\alpha_q, \beta_q, \gamma_q > q$ depending exclusively on $q$. 
This result shows that 
in a large deviation regime, the tails of independent sums are
dominated by a polynomial term with a decay rate
of their highest existing absolute moment.
In a small deviation regime however, they are
dominated by a subgaussian term, essentially 
reflecting the central limit theorem.

Over the last years, different variations of bounds of the above type have been derived. 
For example, \cite{rio2017constants} and \cite{fan2017} derive bounds for maxima
of real-valued martingales by giving explicit expressions for
$\alpha_q, \beta_q, \gamma_q$ under 
corresponding conditional moment assumptions,
\cite{rio2017dependent} proves a bound 
for dependent sums under mixing assumptions,
and \citet{jirak2025} give a Fuk--Nagaev bound for 
independent sums of matrices.
For independent sums of random variables in normed spaces, \citet{Yurinsky1995}
proves a Fuk--Nagaev bound for the centered term
$\norm{S_n} - \E[\norm{S_n}]$ 
under higher moment assumptions for $q \in \N$, $q \geq 3$,
without providing explicit constants, but gives an inconsistent
proof as noted by \citet{mollenhauer2025regularized}.
\citet{marchina2021} derives a bound
for centered empirical processes and provides explicit constants.
\citet{einmahl} prove a result in Banach spaces
for $\max_{1 \leq k \leq n} \norm{S_k} - (1+\eta) \E[\norm{S_n}]$
with $\eta > 0$ and undetermined constants.

It is known that for Banach spaces with \emph{smooth norms}
\citep{pisier1971}
sharp analogues of classical results such as
the Hoeffding and Bernstein inequalities
can be given and the centering term $\E[\norm{S_n}]$ 
can in fact be removed
\citep{Pinelis1986Remarks,Yurinsky1995,Pinelis92,Pinelis94}.
Based on this observation,
\cite{mollenhauer2025regularized} recently proved a Hilbert space
version of the Fuk--Nagaev 
bound by \cite{Yurinsky1995} in which the centering
term is removed, but the constants remain undetermined
and the bound is only valid for $q \in \N$, $q \geq 3$.

In these notes, we derive a Fuk--Nagaev 
bound for \emph{uncentered} norms of 
martingales in smooth Banach spaces that includes explicit
constants, reconciling results
for real-valued martingales with the theory of infinite-dimensional tail bounds.
Up to a dependence on the 
modulus of smoothness of the Banach space,
we recover the optimal constants in the exponential term
derived by \citet{rio2017constants} from the real-valued
case. The general structure of the 
dependence on the smoothness
appears to be sharp \citep{Pinelis94}.
Up to smoothness, we 
also recover the constant of the polynomial
term derived by \cite{rio2017constants} and
note that the polynomial term itself is known to be
asymptotically
sharp \citep[e.g.][]{mollenhauer2025regularized}.
As an application, we derive a McDiarmid-type inequality
for Banach-space valued functions of independent
random variables with higher conditional moments
based on the classical martingale technique. This inequality
extends unbounded versions of the McDiarmid inequality 
with subgaussian and subexponential conditional distributions 
\citep{kontorovich2014,maurerpontil2021}
to heavy-tailed distributions. We highlight that the assumptions
about the conditional higher moments are easily verified for
Hölder continuous functions.

\section{Main Result}

All random variables considered in these notes will be defined on a
common probability space $(\Omega, \sigalg, \P)$.
We will not discuss the intricacies of measurability and integrability
of Banach-spaced random variables here and refer the reader to 
\citet{ledoux1991} for more details. We simply assume that
all random variables in question are 
almost surely separably valued (or that the Banach space itself is separable) in order to avoid measurability issues.

\begin{definition}[$(2,D)$-smoothness]
    Let $D \geq 1$.
    A Banach space $\cX$ equipped with the norm $\norm{\cdot}$ is called
    \emph{$(2,D)$-smooth}, if
    \begin{equation*}
        \norm{x + y}^2 
        +
        \norm{x - y}^2 
        \leq
        2 \norm{x}^2
        +
        2 D \norm{y}^2
        \quad 
        \text{for all } 
        x, y \in \cX.
    \end{equation*}

We refer the reader to \citet{pisier1971,pisier2016} for a general analysis
of smoothness in the context of Banach-space valued martingales
and to \cite{neerven2020} for examples.
If $\cX$ is a Hilbert space, then it is $(2,1)$-smooth.
If  $\cY$ is a $(2,D)$-smooth Banach space, 
then the Bochner--Lebesgue space 
$L^p(\Omega, \sigalg, \mu; \cY)$ is $(2,D\sqrt{p-1})$-smooth
for all $p \geq 2$ and any measure space $(\Omega, \sigalg, \mu)$.
In particular, the Lebesgue space 
$L^p(\Omega, \sigalg, \mu; \R)$ is $(2,\sqrt{p-1})$-smooth, see also \cite{Pinelis94}.

\end{definition}

\filbreak
\begin{theorem}[Fuk--Nagaev inequality, confidence bound version]
\label{thm:main}

Let $(\X, \norm{\cdot})$ be a $(2, D)$-smooth Banach space.
Let $(M_i)_{0 \leq i \leq n}$ be a martingale in $\X$ adapted to
a nondecreasing filtration $(\sigalg_i)_{0 \leq i \leq n}$
with $M_0 = 0$.
Define the conditional expectation
operators $\E_i[\, \cdot \, ] := \E[ \, \cdot \, \vert \sigalg_i ]$ for $i=0,\ldots,n$
and martingale differences $\xi_i := M_i - M_{i-1}$ for $i=1,\ldots,n$.
Assume
\begin{equation}
    \label{eq:mom}
    \sigma^2 :=
    \esssup \sum_{i=1}^n \E_{i-1} [ \norm{\xi_i}^2 ]
    < \infty
    \quad
    \text{ and }
    \quad
    C^q_q := 
    \esssup \sum_{i=1}^n \E_{i-1} [ \norm{\xi_i}^q ] < \infty
\end{equation}
for some $q > 2$.
Then for all $u \in (0,1)$ we have
\begin{equation} \label{eq:main_bound}
    \P \left[
        \max_{i \in \{1, \dots n \}} \norm{M_n}
        \leq D \sigma \sqrt{2\log(2/u)} + c_{q,D}C_q \left(\frac{2}{u}\right)^{1/q}
    \right] 
    \geq 1-u.
\end{equation}
with the constant
\begin{equation} \label{eq:c_q,D_constant}
    c_{q,D}  := 
    \frac{1}{2q} 
    + \min\{1/q, 1/5\} + 1 
    + \mathbbm{1}_{q > 3}\frac{D^2 q}{3}.
\end{equation}    
\end{theorem}

\cref{thm:main} essentially extends the result of
\citet[][Corollary 3.2]{rio2017constants}
for real martingales to the Banach space
case.
We prove \cref{thm:main} in \cref{sec:proof_main}.

\subsection{Rearranging to a tail bound}
Directly inverting the upper bound
given in \eqref{eq:main_bound} in order to
derive a tail bound requires solving a
transcendental equation. Instead, we split
the sum and handle both terms separately.
Let $t > 0$. We
want to find $u^*$ such that simultaneously, we have
\begin{equation*}
    D \sigma \sqrt{2\log(2/u^*)} \leq \frac{t}{2}
    \quad
    \text{and}
    \quad
    c_{q,D}C_q \left(\frac{2}{u^*}\right)^{1/q} 
    \leq \frac{t}{2},
\end{equation*}
which by \cref{thm:main} implies 
$\P\left[ 
    \max_{i \in \{1, \dots n \}} \norm{M_n}
    > t
\right] \leq u^*$.
We can rearrange both terms and determine
the simultaneous conditions
\begin{equation}
    \label{eq:tail_rearranged}
    u^* \geq 2 \exp 
    \left( 
    - \frac{t^2}{8 D^2 \sigma^2}
    \right)
    \quad
    \text{and}
    \quad
    u^* \geq
    2 \left(
    \frac{2 c_{q,D}C_q }{t}
    \right)^q.
\end{equation}
We can now simply set $u^*$ to be
the sum of the two terms on the right hand sides in
\eqref{eq:tail_rearranged}. We
have shown the following.

\begin{corollary}[Fuk--Nagaev inequality, tail bound version]
\label{cor:fn_tail_bound}
Under the assumptions of \cref{thm:main}, we have
    \begin{equation*}
    \P\left[ 
    \max_{i \in \{1, \dots n \}} \norm{M_n}
    > t
    \right] 
    \leq 
    2 \left(
    \frac{2 c_{q,D}C_q }{t}
    \right)^q
    +
    2 \exp 
    \left( 
    - \frac{t^2}{8 D^2 \sigma^2}
    \right).
\end{equation*}
for all $t > 0$.
\end{corollary}

\subsection{Independent sums}
For convenience, we consider the special case that $\xi_1,\ldots,\xi_n$ 
are independent, identically distributed and centered
random variables taking values 
in $\X$ and fulfill
\begin{equation*}
    \sigma^2  :=  \E[ \norm{\xi_i}^2 ] < \infty 
    \quad \text{and} \quad 
     C^q_q  := \E[ \norm{\xi_i}^q ] < \infty,
\end{equation*}
instead of the assumption \eqref{eq:mom}. 
Then for all $u\in(0,1)$, we 
straightforwardly obtain the bound
\begin{equation}
    \label{eq:fn_independent}
    \P\left[ \max_{k=1,\ldots,n} 
        \left \Vert \frac1n \sum_{i=1}^k \xi_i \right\Vert
        \leq D \sigma 
            \sqrt{ \frac{2\log(2/u)}{n} } 
            + c_{q,D}C_q \left(\frac{2}{u n^{q-1} }\right)^{1/q}
    \right] 
    \geq 1-u
\end{equation}
by applying \cref{thm:main} to the martingale given by
$M_k = \sum_{i=1}^k \xi_k$, $k=1,\ldots,n$
with respect to the naturally induced filtration.
The bound \eqref{eq:fn_bound} 
generalizes and sharpens the Hilbert space bound
by \citet[][Corollary 5.3]{mollenhauer2025regularized}
by giving precise constants and extending it to all real $q > 2$.


\section{A McDiarmid bound for heavy-tailed functions}

As an application of \cref{thm:main},
we give a bound for Banach-space valued functions
of independent random variables that generalizes
the well-known bound by McDiarmid. Instead of the
classical bounded differences assumption \citep{mcdiarmid}, we 
only assume a uniform control of the individual higher moments
of the conditional distributions in each coordinate.
Similar inequalities have been derived by \citet{kontorovich2014}
and \citet{maurerpontil2021} under assumptions on
conditional Orlicz norms, 
allowing to weaken the bounded difference 
condition and giving exponential concentration.
In both cases, the authors require additional tools
to incorporate the componentwise Orlicz norms into the
final composite function.
In our case, we can directly work with the classical 
martingale approach
and apply \cref{thm:main}, as assumptions about finite
higher moments are
naturally compatible with this type of decomposition.
Moreover, our assumptions can be easily verified
for Lipschitz and Hölder continuous functions.

We consider independent random 
variables $Z_1, \dots, Z_n$
defined on a common probability space $(\Omega, \sigalg, \P)$ 
taking values
in measurable spaces
$(\mathcal{Z}_1, \Xi_1), \dots, (\mathcal{Z}_n, \Xi_n)$.
Let $\mu_1, \dots, \mu_n$ denote the corresponding distributions
of $Z_i, \dots, Z_n$.
We write $\mathcal{Z} := \prod_{i=1}^n \mathcal{Z}_1$
for the product space and equip it with the product $\sigma$-field.
Furthermore, we introduce the abbreviations
$\bz := (z_1, \dots, z_n) \in \mathcal{Z}$ and 
$\bZ := (Z_1, \dots, Z_n)$ and
for a measurable map $f$ on $\mathcal{Z}$, we introduce the 
\emph{conditional versions} of $f$ as
\begin{equation*}
    f_i(\bz;\bZ)
    :=
    f(z_1, \dots z_{i-1}, Z_i, z_{i+1}, \dots, z_n ),
\end{equation*}
which allow to conveniently investigate the
influence of $\mu_i$ on the $i$-th coordinate of $f$ 
when all other arguments are fixed.
%
%
\begin{corollary}[McDiarmid inequality for heavy-tailed functions]
Let $(\X, \norm{\cdot})$ be a $(2, D)$-smooth Banach space. 
Consider a measurable function $f: \mathcal{Z} \rightarrow \X$ such that
\begin{equation}
    \label{eq:mcdiarmid_constants}
    \begin{split}
    \sigma^2 &:=
    \sum_{i=1}^n
    \sup_{\bz \in \mathcal{Z}}
    \E[ \norm{ f_i(\bz;\bZ) - \E[ f_i(\bz;\bZ)] }^2 ] < \infty
    \quad \text{and} \quad \\
    C^q_q &:=
    \sum_{i=1}^n
    \sup_{\bz \in \mathcal{Z}}
    \E[ \norm{ f_i(\bz;\bZ) - \E[ f_i(\bz;\bZ)] }^q ] < \infty
    \quad \text{for some } q>2.
    \end{split}
\end{equation}    
Then for all $u\in(0,1)$, we have
\begin{equation*}
     \P \left[
        \left\| f(Z_1,\ldots,Z_n) - \E[f(Z_1,\ldots,Z_n)] \right\|_\X
        \leq D \sigma \sqrt{2\log(2/u)} + c_{q,D}C_q \left(\frac{2}{u}\right)^{1/q}
    \right] 
    \geq 1-u,
\end{equation*}
where the constant $c_{q,D}$ is given by \eqref{eq:c_q,D_constant}.
\end{corollary}
\begin{proof}
Let $\sigalg_i \subseteq \sigalg$ denote the 
$\sigma$-field induced by $Z_1, \dots, Z_i$ and
set $\sigalg_0 := \{ \emptyset, \Omega \}$.
We define the Doob martingale
\begin{equation*}
    M_i := \E[ f(\bZ) - \E[f(\bZ)] \mid \sigalg_i ],
    \quad
    i = 1, \dots, n
\end{equation*}
and $M_0 = 0$
and corresponding increments 
$\xi_i := M_i - M_{i-1} 
= 
\E[ f(\bZ) \mid \sigalg_i ] - \E[ f(\bZ) \mid \sigalg_{i-1}]$.
Using a well-known representation of the conditional expectations based on Fubini's Theorem and the Doob-Dynkin Lemma, we get
\begin{align*}
    \xi_i &=
    \int
    f(Z_1, \dots, Z_i, z_{i+1}, \dots z_n ) \,
    \dd \mu_{i+1}(z_{i+1}) \otimes \dots \otimes \mu_n(z_n) \\
    & \qquad - \int
    f(Z_1, \dots, Z_{i-1}, z_{i}, \dots z_n ) \,
    \dd \mu_{i}(z_{i}) \otimes \dots \otimes \mu_n(z_n) \\
    &=
    \int
    \left(
    \int
    f(Z_1, \dots, Z_i, z_{i+1}, \dots z_n )
    -
    f(Z_1, \dots, Z_{i-1}, z_{i}, \dots z_n ) \,
    \dd \mu_i(z_i)
    \right)
    \dd \mu_{i+1}(z_{i+1}) \otimes \dots \otimes \mu_n(z_n),
\end{align*}
and we specifically notice that the two
evaluations of $f$ in the parentheses only differ in
the $i$-th coordinate.
Based on the above
expression, we now find that for all $p \geq 1$,
we have
\begin{align*}
    &\esssup
    \E\left[\norm{\xi_i}^p \mid \sigalg_{i-1} \right] \\
    &\qquad\leq
    \esssup
    \int 
    \left\Vert f(Z_1, \dots, Z_{i-1}, \hat z_i, z_{i+1}, \dots z_n )
    - \int
    f(Z_1, \dots, Z_{i-1}, z_{i}, \dots z_n )
    \dd \mu_i(z_i)
    \right\Vert^p \,
     \, \dd \mu_i(\hat z_i) \\
    & \qquad  
    \leq
    \sup_{\bz \in \mathcal{Z}}
    \E[ \norm{ f_i(\bz;\bZ) - \E[ f_i(\bz;\bZ)] }^{p} ].
\end{align*}
We now realize that under the conditional moment assumptions
\begin{equation*}
    \sigma^2 :=
    \sum_{i=1}^n
    \sup_{\bz \in \mathcal{Z}}
    \E[ \norm{ f_i(\bz;\bZ) - \E[ f_i(\bz;\bZ)] }^2 ]
    \quad \text{and} \quad
    C^q_q :=
    \sum_{i=1}^n
    \sup_{\bz \in \mathcal{Z}}
    \E[ \norm{ f_i(\bz;\bZ) - \E[ f_i(\bz;\bZ)] }^q ],
\end{equation*}
the martingale $(M_i)_{0 \leq i \leq n}$ satisfies 
the assumptions of \eqref{thm:main},
so the result follows by using \cref{thm:main}.

\end{proof}

\subsection{Hölder continuous functions}

We briefly give basic conditions under which the
conditional moment assumption
\eqref{eq:mcdiarmid_constants} can be verified.
We consider now a function $f: \mathcal{Z} \to \cX$,
where $\cX$ is a $(2,D)$-smooth Banach space
and $\mathcal{Z} := \prod_{i=1}^n \mathcal{Z}_1$ is the product
of metric spaces
$(\mathcal{Z}_1, d_1), \dots, (\mathcal{Z}_n, d_n)$
equipped with the metric
$d := \sum_{i=1}^n d_i$ and corresponding Borel $\sigma$-field.
Assume that $f$ satisfies the Hölder condition
\begin{equation}
    \label{eq:lipschitz}
    \norm{f(\bz) - f(\hat \bz)} 
    \leq L \, d(\bz, \hat \bz)^\alpha \quad \forall \bz,\hat\bz \in \bZ
\end{equation}
for some fixed $\alpha \in (0,1]$ and $L \geq 0$.
We now find that for all $p \geq 1$,
we have
\begin{align*}
    \sup_{\bz \in \mathcal{Z}}
    \E[ \norm{ f_i(\bz;\bZ) - \E[ f_i(\bz;\bZ)] }^p ]
    &\leq
    \sup_{\bz \in \mathcal{Z}}
    \E_{\bZ, \hat \bZ}[ \norm{ f_i(\bz;\bZ) -  f_i(\bz;\hat \bZ) }^p ] \\
    &\leq L^p \,
    \E_{Z_i, \hat Z_i}[ d_i(Z_i, \hat Z_i)^{\alpha p} ],
\end{align*}
where $\hat\bZ$ is an independent copy of $\bZ$.
In the first step we use Jensen's inequality
and in the second step we use \eqref{eq:lipschitz} 
together with the definition of the
metric $d$ and the fact that the evaluations of $f_i$
only differ in the $i$-th coordinate.
We can now define
\begin{equation}
    \label{eq:mcdiarmid_constants_lip}
    \sigma^2 :=
    L^2
    \sum_{i=1}^n
    \E \left[
    d_i(Z_i, \hat Z_i)^{2 \alpha}
    \right]
    \quad \text{and} \quad
    C^q_q :=
    L^q
    \sum_{i=1}^n
    \E \left[
    d_i(Z_i, \hat Z_i)^{\alpha q}
    \right]
\end{equation}
and the calculation above verifies that the martingale $M_n$
satisfies the assumptions of \cref{thm:main}
with corresponding constants given by 
\eqref{eq:mcdiarmid_constants_lip}
in case they are finite.
%
%
In fact, the quantities $\sigma^2$ and $C^q_q$ can be bounded
further depending on the setting.
If for example the spaces $\mathcal{Z}_i$
are normed spaces with norms $\norm{\cdot}_i$,
we may simplify the assumptions since
for every $p \geq 1$, we have
\begin{equation*}
        \E \left[
        d_i(Z_i, \hat Z_i)^p
        \right]
        \leq
        2^p \, \E \left[
        \norm{Z_i}_i^p
        \right],
\end{equation*}
where we use the triangle inequality, independence,
and that $(x+y)^p \leq 2^{p-1}(x^p + y^p)$ for $x,y>0$.

\section[Proof of Theorem~\ref{thm:main}]{Proof of \cref{thm:main}}
\label{sec:proof_main}

The overall proof strategy follows \citet[][Section 3]{rio2017constants}.
It derives the optimization of a Chernoff bound corresponding to a truncated martingale
in terms of its quantile function which is conveniently expressed as an inverse
Legendre transformation.

The crucial step that allows to transfer the original proof by \cite{rio2017constants}
to the Banach space setting is a bound on the cumulant generating function of the norm of vector-valued martingales
which we obtain by reformulating the classic results by \cite{Pinelis94}.
This bound is structurally identical to the bound of the cumulant
generating function supplied by \cite[][Lemma 3.4]{rio2017constants},
allowing to perform similar optimization steps.
This extends the standard Chernoff bound optimization arguments
Hoeffding's and Bernstein's inequalities
in Banach spaces by \cite{Pinelis94} to the Fuk--Nagaev inequality.

Let $X$ be a real-valued random variable. We define 
the \emph{quantile function}\footnote{
With this definition, $Q_X(u)$ is the largest $1-u$ quantile of $X$.
Note that the quantile function is more commonly defined in terms
of the cumulative distribution function instead of the tail function.
}
\begin{equation*}
    Q_X(u) := \inf\{ t \in \R \mid \P[ X > t ] < u \}, \quad u \in (0,1].
\end{equation*}
Note that $Q_X$ is nonincreasing in $u$.
Furthermore, if $U$ is a uniform random variable on $[0,1]$, then
$Q_X(U)$ is distributed according to $X$.
We additionally define the nonincreasing
\emph{integrated quantile function}
\begin{equation*}
    Q^1_X(u) 
    := 
    \frac{1}{u} \int^u_0 Q_X(s) \, \dd s
    =
    \E[ X \mid X \geq Q(u) ], \quad u \in (0,1]
\end{equation*}
and for $u\in (0,1]$, we also define the following moment expression
\begin{equation*}
    Q_X^\infty(u) :=
    \inf_{t \in (0, \infty)}
    t^{-1} \, \log\left( \E [ \exp(tX) / u ] \right).
\end{equation*}
The function $Q_X^1$ is also called the
\emph{conditional value-at-risk}, see \citet{Pflug2000} and \citet{Rockafellar2000}. We also refer the reader
to \citet{Pinelis2014} for a thorough analysis
of the properties of the functions defined above.

\subsection{Approximation by truncated martingale}

We realize that without loss of generality,
we can consider the simplified assumption
\begin{equation}
    \label{eq:assumptions}
    \sigma^2 :=
    \esssup \sum_{i=1}^n \E_{i-1} [ \norm{\xi_i}^2 ]
    < \infty
    \quad
    \text{ and }
    \quad
    \esssup \sum_{i=1}^n \E_{i-1} [ \norm{\xi_i}^q ] \leq 1.
\end{equation}
The general case as in \cref{thm:main} follows by considering the
scaled martingale $M_n / C_q$.

For a threshold $L > 0$, define the \emph{level-$L$ truncations} of the random variables 
$\xi_i$ as $\tilde{\xi_i} := \mathbbm{1}_{\norm{\xi_i} \leq L} \xi_i$.
The corresponding truncated martingale is given by $\tilde{M_n} = \tilde{\xi_1} + \dots, \tilde{\xi_n}$.
Following \citet{rio2017constants},
we perform the basic Chernoff argument 
that allows to bound the tails of $M_n$ 
conveniently in terms 
of quantile functions, see \cref{app:quantile_functions}
for the definition and basic properties.

We use the shorthand notation 
$M_n^* := \max_{i \in \{1, \dots n \}} \norm{M_n}$.
For all $u\in(0,1)$ we now get
\begin{align}
    \label{eq:quantile_function_bound}
    Q_{M_n^*}(u) \leq
    Q_{\norm{M_n}}^1(u) 
    &\leq
    Q^1_{\norm{M_n - \tilde{M}_n} + \norm{\tilde{M}_n}}(u) \nonumber \\
    &\leq
    Q^1_{\norm{M_n - \tilde{M}_n}}(u)
    +
    Q^1_{\norm{\tilde{M}_n}}(u) \nonumber \\
    &\leq
    Q^1_{\norm{M_n - \tilde{M}_n}}(u)
    +
    Q^\infty_{\norm{\tilde{M}_n}}(u),
\end{align}
where we used in turn \cref{lem:submartingale_inequality}, 
\cref{lem:quantile_monotonicity},
\cref{lem:subadditivity}, and
\cref{lem:quantile_bounds}.

\subsection{Bounding the approximation error}
Let $u\in(0,1)$.
Following \citet{rio2017constants}, we can proceed to bound
\begin{align*}
    Q^1_{\norm{M_n - \tilde{M}_n}}(u)
    &\leq
    \frac{1}{u} \E[ \norm{M_n - \tilde{M}_n } ] 
    &&
    \text{(\cref{lem:variational_formulation})} \\
    & \leq
    \frac{1}{u}
    \sum_{i = 1}^n 
    \E \left[ \E_{i-1} [ \norm{\xi_i - \tilde{\xi}_i }  ] \right]
    &&
    \text{(triangle inequality)} \\
    &\leq
    \frac{1}{u}
    \sum_{i = 1}^n 
    \E \left[ 
         \int^\infty_L \P[ \norm{\xi_i - \tilde{\xi}_i }  > s  \mid  \sigalg_{i-1} ] \, \dd s 
    \right]
    &&
    \text{(truncation \& tail integration)} \\
    &\leq
    \frac{1}{u q L^{q-1}}
    \sum_{i = 1}^n 
    \E \left[ 
        \int^\infty_L q s^{q-1} \P[ \norm{\xi_i }  > s  \mid  \sigalg_{i-1}] \, \dd s 
    \right]
    &&
    \text{(multiplication by $1$)} \\
    &\leq
    \frac{1}{u q L^{q-1}}
    \sum_{i = 1}^n 
    \E \left[ 
        \int^\infty_0 \P[ \norm{\xi_i }^q > s  \mid  \sigalg_{i-1} ] \, \dd s 
    \right]
    &&
    \text{(integral substitution)} \\
    & =\frac{1}{u q L^{q-1}}
    \sum_{i = 1}^n 
    \E \left[ 
        \E_{i-1} [ \norm{\xi_i}^q ]
    \right]
    && \text{(tail integration)} \\
    & \leq 
    \frac{1}{u q L^{q-1}} \esssup \sum_{i=1}^n \E_{i-1} [ \norm{\xi_i}^q ]
\end{align*}
from which we finally obtain 
\begin{equation}
    \label{eq:approximation bound}
    Q^1_{\norm{M_n - \tilde{M}_n}}(u)
    \leq
    \frac{1}{q u^{1/q}}
\end{equation}
by 
\eqref{eq:assumptions}
and choosing the truncation level $L = u^{-1/q}$.

\subsection{Bounding the truncated martingale}
We now need to bound  $Q^\infty_{\norm{\tilde{M}_n}}(u)$.
We also need the following classical result which allows to 
derive sharp bounds for
tail probabilities of martingales in smooth Banach spaces.

\begin{lemma}[Exponential moment bound, \citealt{Pinelis94}, proof of Theorem 3.1]
Let $\X$ be a $(2, D)$-smooth Banach space.
Let $(\tilde M_i)_{0 \leq i \leq n}$ be a martingale in $\X$ adapted to
the nondecreasing filtration $(\sigalg_i)_{0 \leq i \leq n}$
and set $\tilde M_0 = 0$.
Given any fixed $t > 0$, consider the process $(G_i)_{0 \leq i \leq n}$ defined by $G_0 := 1$
and
\begin{equation*}
    G_i := \cosh(t \norm{\tilde M_i}) / \prod_{i=1}^n (1 + e_i), 
    \quad
    i \in \{1, \dots, n\}
\end{equation*}
with
\begin{equation*}
    e_i :=  D^2 
         \E_{i-1}\left[
         e^{t \norm{\tilde \xi_i}}
         - 1
         - t \norm{\tilde \xi_i}
         \right].
\end{equation*}
Then  $(G_i)_{0 \leq i \leq n}$ is a nonnegative supermartingale.
\end{lemma}

Therefrom, we obtain
\begin{equation*}
   \E\left[ \cosh(t \norm{\tilde M_n}) / 
   \left\| \prod_{i=1}^n (1 + e_i) \right\|_{L^\infty(\P)} \right]
   \leq
   \E[G_n] 
   \leq
   \E[G_0] = 1.
\end{equation*}
The term $ \left\| \prod_{i=1}^n (1 + e_i) \right\|_{L^\infty(\P)}$ is deterministic, 
implying
\begin{equation}
    \label{eq:pinelis_cos_bound}
    \E\left[ \cosh(t \norm{\tilde M_n}) \right]
   \leq  \left\| \prod_{i=1}^n (1 + e_i) \right\|_{L^\infty(\P)},
\end{equation}
from which we deduce
\begin{align}
    \label{eq:pinelis_mgf_bound}
    \E[ \exp(t \norm{\tilde M_n}) ]
    &\leq
    2 \,\E[ \cosh(t \norm{\tilde M_n}) ] 
    \leq
    2 \left\| \prod_{i=1}^n (1 + e_i) \right\|_{L^\infty(\P)} \nonumber \\
    &\leq
    2\prod_{i=1}^n (1 + \norm{ e_i}_{L^\infty(\P)}  ) 
    \leq
    2\exp\left( \sum_{i=1}^n \norm{ e_i}_{L^\infty(\P)} \right),
\end{align}
where we use
$\cosh(x) = (e^x + e^{-x})/2$, the bound from \eqref{eq:pinelis_cos_bound},
the submultiplicativity and triangle
inequality of the $L^{\infty}(\P)$-norm 
and $1 + x \leq e^x$ for all $x \in \R$.

We now note that $\norm{\tilde \xi_i} \leq \norm{\xi_i}$ almost surely 
implies that the assumptions (\ref{eq:assumptions}) are also valid for the truncated random
random variables $\tilde \xi_i$, which means that we have
\begin{equation}
    \label{eq:assumptions_truncated}
    \esssup \sum_{i=1}^n \E_{i-1} [ \norm{\tilde \xi_i}^2 ]
    < \sigma^2
    \quad
    \text{ and }
    \quad
    \esssup \sum_{i=1}^n \E_{i-1} [ \norm{\tilde \xi_i}^q ] \leq 1.
\end{equation}
The following moment bound is precisely the assertion of
\citet[][Proposition 3.5]{rio2017constants} applied
to the real-valued random variables $\norm{\tilde \xi_i}$.

\begin{lemma}[Moment bound, \citealp{rio2017constants}, Proposition 3.5]
    \label{lem:rio_moment_bounds}
    Let $\tilde \xi_1, \dots, \tilde \xi_n$ be a finite sequence of random 
    variables in a normed space adapted to
    the nondecreasing filtration $(\sigalg_i)_{0 \leq i \leq n}$ satisfying
    (\ref{eq:assumptions_truncated}) and $\norm{\tilde \xi_i} \leq L$
    almost surely. Then we have
    \begin{equation*}
        \esssup \sum_{i=1}^n \E_{i-1} [ \norm{\tilde \xi_i}^k ]
        \leq \sigma^{2 (q-k)/ (q-2)},
        \quad
        k \in [2 , q]
    \end{equation*}
    as well as
    \begin{equation*}
        \esssup \sum_{i=1}^n \E_{i-1} [ \norm{\tilde \xi_i}^k ]
        \leq L^{k-q},
        \quad
        k \geq q.
    \end{equation*}
\end{lemma}

With this result, we use the bound of the moment
generating function given by (\ref{eq:pinelis_mgf_bound}) to
the cumulant generating function bound
\begin{align}
    \log \left( \frac{1}{2}  \E[ \exp(t \norm{\tilde M_n}) ] \right)
    &\leq
    \sum_{i=1}^n \norm{ e_i}_{L^\infty(\P)} \nonumber \\
    &=
    D^2 
    \sum_{i=1}^n 
    \left\Vert
    \E_{i-1}\left[
    e^{t \norm{\tilde \xi_i}}
    - 1
    - t \norm{\tilde \xi_i}
    \right]
    \right\Vert_{L^\infty(\P)} \nonumber \\
    &
    \leq
    D^2 
    \sum_{i=1}^n 
    \sum_{k = 2}^\infty
    \left\Vert
    \E_{i-1}\left[
    \norm{\tilde\xi_i}^k
    \right] 
    \right\Vert_{L^\infty(\P)} \frac{t^k}{k!} \nonumber
    \\
    & \leq
    D^2 
    \left(
    \underbrace{
    \sigma^2 \frac{t^2}{2}
    }_{=:\ell_0(t)}
    +
    \underbrace{
    \sum_{ 2 < k < q}
    \sigma^{2 (q-k)/ (q-2)} \frac{t^k}{k!}
    }_{=:\ell_1(t)}
    +
    \underbrace{
    L^{-q} \sum_{k \geq q} \frac{L^kt^k}{k!}
    }_{=:\ell_2(t)}
    \right) \label{eq:cgf_bound}
\end{align}
for all $t > 0$.

\subsection{Chernoff bound optimization}

We proceed to bound $Q^\infty_{\norm{\tilde{M}_n}}(u)$ 
in terms of the \emph{inverse Legendre transform} of
the cumulant generating function given by \eqref{eq:cgf_bound}.
For this, we closely follow the arguments from the proof of \cite[Theorem~3.1(a)]{rio2017constants}.

\begin{definition}[Inverse Legendre transform]
    For a convex function $\psi: [0, \infty) \to [0, \infty)$
    with $\psi(0) = 0$,
    we define the \emph{inverse Legendre transform} 
    as
    \begin{equation*}
        [\mathcal{T}\psi](x) = 
        \inf_{t > 0}
        \{
        t^{-1} (\psi(t) + x)
        \},
        \quad
        x \in [0,\infty].
    \end{equation*}
\end{definition}

We refer to \citet[][ Annexes A \& B]{rio2017dependent}
for some background on the inverse Legendre transform
in the context of the classical Chernoff bound.

We define 
$\ell(t) :=
\log\left( \frac12 \E [ \exp(t\norm{\tilde{M}_n}) ] \right)$.
We have
\begin{align*}
    Q_{\norm{\tilde{M}_n}}^\infty(u) &=
        \inf_{t \in (0, \infty)}
            t^{-1} \, \log\left( \E [ \exp(t\norm{\tilde{M}_n}) / u ] \right)
        && \text{(Definition of $Q_{\norm{\tilde{M}_n}}^\infty(u)$)} \\
    & =
        \inf_{t \in (0, \infty)}
            t^{-1} \, \log\left( \frac12 \E [ \exp(t\norm{\tilde{M}_n}) / (u/2) ] \right) & \\
    & = 
        \inf_{t \in (0, \infty)}
        \frac{\log\left( \frac12 \E [ \exp(t\norm{\tilde{M}_n}) ] \right) +  \log(2/u)}{t} 
        && \text{($\log(ab)=\log(a) + \log(b)$)} \\
    &= 
        \mathcal{T}[\ell](\log(2/u))
        && \text{(Definition of $\mathcal{T}$)} \\
    & \leq \mathcal{T}[D^2(\ell_0 + \ell_1 + \ell_2)](\hat x)
    && \text{(Monotonicity of $\mathcal{T}$ and \eqref{eq:cgf_bound})} ,
\end{align*}
where we defined $\hat x = \log(2/u)$ for brevity and
$\ell_0, \ell_1, \ell_2$ are defined in \eqref{eq:cgf_bound}.
Using the subadditivity of $\mathcal{T}$ in its functional argument \citep[Proposition~2.5(i)]{rio2017constants}, we have
\begin{align*}
     \mathcal{T}[D^2(\ell_0 + \ell_1 + \ell_2)](\hat x) & \leq  \mathcal{T}[D^2\ell_0](\hat x) +  \mathcal{T}[D^2(\ell_1 + \ell_2)](\hat x) \text{ and } \\
     \mathcal{T}[D^2(\ell_0 + \ell_1 + \ell_2)](\hat x) & \leq  \mathcal{T}[D^2(\ell_0 + \ell_1)](\hat x) +  \mathcal{T}[D^2\ell_2](\hat x),
\end{align*}
and hence
\begin{equation*}
     \mathcal{T}[D^2(\ell_0 + \ell_1 + \ell_2)](\hat x) \leq \min
     \left\{ 
      \mathcal{T}[D^2\ell_0](\hat x) +  \mathcal{T}[D^2(\ell_1 + \ell_2)](\hat x),
      \mathcal{T}[D^2(\ell_0 + \ell_1)](\hat x) +  \mathcal{T}[D^2\ell_2](\hat x)
     \right\}. 
\end{equation*}

We now proceed by bounding the individual terms occurring in the right hand side.
\paragraph{Step 1: Bounding $\mathcal{T}[D^2\ell_2](\hat x)$.}
We define
\begin{equation*}
    \psi_q(t) := \sum_{k \geq q} \frac{t^k}{k!},
\end{equation*}
so we may reformulate
\begin{equation*}
    \ell_2(t) =  L^{-q} \sum_{k \geq q} \frac{L^kt^k}{k!} = L^{-q}\psi_q(Lt).
\end{equation*}
We now get
\begin{align*}
     \mathcal{T}[D^2\ell_2](\hat x) & = \inf_{t>0} \frac{D^2\ell_2(t)+\hat x}{t} \\
        & \leq \frac{D^2\ell_2(\hat x /L) + \hat x}{\hat x / L} 
            & \text{(Choose $t=\hat x / L$)} \\
        & = \frac{L}{\hat x} \left(D^2 L^{-q}\psi_q(L \cdot \hat x/L) + \hat x \right) \\
        & = D^2 L^{1-q} \hat x^{-1}\psi_q(\hat x) + L \\
        & \leq D^2 L^{1-q} e^{\hat x}\min\{1/q, 1/5\} + L
            & \text{(\citealp[Lemma~3.6]{rio2017constants})} \\
        & = \left( D^2 L^{-q}e^{\hat x}\min\{1/q, 1/5\} + 1\right)\cdot L.
\end{align*}
We now choose the specific truncation level $L=e^{\hat x / q}$. 
With this choice of $L$ we get
\begin{equation*}
    \mathcal{T}[D^2 \ell_2](\hat x) \leq \left( D^2 L^{-q}e^{\hat x}\min\{1/q, 1/5\} + 1\right)\cdot L = (1+D^2\min\{1/q, 1/5\})\cdot L = \alpha_{q,D} L,
\end{equation*}
where we define $\alpha_{q,D} := D^2\min\{1/q, 1/5\} + 1$.
\paragraph{Step 2: Bounding $\mathcal{T}[D^2\ell_0](\hat x)$.}
Using the definition of $\ell_0$, we have
\begin{equation*}
    \mathcal{T}[\ell_0](\hat x) = \inf_{t>0} \frac{D^2\ell_0(t)+\hat x}{t} = \inf_{t>0} \frac{D^2\sigma^2 \frac{t^2}{2}+\hat x}{t}.
\end{equation*}
Define $f(t)=\frac{\sigma^2 \frac{t^2}{2}+\hat x}{t}$. An elementary calculation 
(setting $f'(t)=0$ and solving for $t$) shows that this function is minimized for $t=2\hat x / (D^2\sigma^2)$ with value $\sqrt{2\hat x}D\sigma$, so we get
\begin{equation*}
    \mathcal{T}[D^2\ell_0](\hat x) = \sqrt{2\hat x}D\sigma.
\end{equation*}
\paragraph{Step 3: Bound for $2 < q\leq 3$.}
Observe that for $q\leq 3$, we have $\ell_1=0$, so
\begin{align*}
      Q_{\norm{\tilde{M}_n}}^\infty(u) & \leq \min\{ 
            \mathcal{T}[D^2\ell_0](\hat x) +  \mathcal{T}[D^2(\ell_1 + \ell_2)](\hat x),
            \mathcal{T}[D^2(\ell_0 + \ell_1)](\hat x) +  \mathcal{T}[D^2\ell_2](\hat x)
            \} \\
        & =  \left( \mathcal{T}[D^2\ell_0](\hat x) +  \mathcal{T}[D^2\ell_2](\hat x) \right) \\
        & \leq \left( \sqrt{2\hat x}D\sigma + \alpha_{q,D} L \right).
\end{align*}
From now on, we assume that $q>3$.
\paragraph{Step 4: Bound for $\mathcal{T}[D^2(\ell_1+\ell_2)](\hat x)$.}
Let $q>3$. We have
\begin{align*}
    \mathcal{T}[D^2(\ell_1+\ell_2)](\hat x) & = \inf_{t>0} \frac{D^2\ell_1(t) + D^2\ell_2(t) + \hat x}{t} \\
        & \leq \frac{L}{\hat x}(D^2\ell_1(\hat x / L) + D^2\ell_2(\hat x / L) + \hat x)
            & \text{(Choose $t=\hat x/ L$)} \\
        & = \frac{L}{\hat x} \cdot (D^2\ell_2(\hat x / L) + \hat x) + \frac{L}{\hat x}D^2\ell_1(\hat x / L) \\
        & \leq \alpha_{q,D} L +  \frac{L}{\hat x}D^2\ell_1(\hat x / L),
\end{align*}
where we bound the term involving $\ell_2$ with the
argument from Step 1.
It remains to bound $ \frac{L}{\hat x}D^2\ell_1(\hat x / L)$.
We recall from Step 1 that we chose $L=e^{\hat x / q}$, so $\hat x / L = \hat x e^{-\hat x / q}$. 
We have
\begin{equation*}
    \frac{\hat x}{q L} = \frac{\hat x}{q} e^{-\hat x / q} \leq \sup_{s>0} se^{-s} = \frac{1}{e}
\end{equation*}
and hence 
\begin{equation*}
    \frac{\hat x}{L} \leq \frac{q}{e}.
\end{equation*}
Furthermore,
\begin{equation*}
    t^{-2}\ell_1(t) = t^{-2}\sum_{2<k<q} \sigma^{\frac{2(q-k)}{q-2}}\frac{t^k}{k!}
    =  \sum_{2<k<q} \sigma^{\frac{2(q-k)}{q-2}}\frac{t^{k-2}}{k!}
\end{equation*}
shows that $\Rp \ni t \mapsto t^{-2}\ell_1(t)$ is increasing,
so we get
\begin{equation*}
    \left(\frac{\hat x}{L}\right)^{-1}\ell_1(\hat x/ L) = \frac{\hat x}{L}   \left(\frac{\hat x}{L}\right)^{-2} \ell_1(\hat x/ L)
    \leq \frac{\hat x}{L}   \left(\frac{q}{e}\right)^{-2} \ell_1(q/e).
\end{equation*}
This allows us to proceed by bounding
\begin{align*}
    \mathcal{T}[D^2(\ell_1+\ell_2)](\hat x) & 
    \leq \alpha_{q,D} L 
        +  D^2 \frac{L}{\hat x} \ell_1(\hat x / L) \\
        & \leq \alpha_{q,D} L 
        + D^2 \frac{\hat x}{L} 
            \left(\frac{q}{e}\right)^{-2} \ell_1(q/e) \\
        & =  \alpha_{q,D} L 
        + D^2\frac{\hat x e}{q} \frac{e}{q} 
            \frac{1}{L} \ell_1(q/e) \\
        & \leq  \alpha_{q,D} L 
        + D^2\frac{\hat x e}{3} \frac{e}{qL} \ell_1(q/e) 
            & \text{($q>3$)} \\
        & \leq \alpha_{q,D} L 
        + D^2\frac{\hat x e}{3} \ell_1(q/e). 
            & \text{($q>3$, $L\geq 1$, $e<3$)}
\end{align*}
\paragraph{Step 5: Bound on $\mathcal{T}[D^2(\ell_0+\ell_1)](\hat x)$.}
We first observe that
\begin{equation*}
    \sigma^{\frac{2(q-k)}{q-2}} = \sigma^{\frac{2(2-k)}{q-2}}\sigma^{\frac{2(q-2)}{q-2}} 
        =  \left(\sigma^{-\frac{2}{q-2}}\right)^{k-2}\sigma^2.
\end{equation*}
For $t>0$, we then obtain
\begin{align*}
    \ell_0(t) + \ell_1(t) & = \sigma^2 \frac{t^2}{2} + \sum_{2<k<q}  \sigma^{\frac{2(q-k)}{q-2}}\frac{t^k}{k!} \\
        & = \sum_{2 \leq k < q} \sigma^{\frac{2(q-k)}{q-2}}\frac{t^k}{k!} \\
        & = t^2 \sum_{2 \leq k < q} \sigma^{\frac{2(q-k)}{q-2}}\frac{t^{k-2}}{k!} \\
        & = \sigma^2 t^2 \sum_{2 \leq k < q}  \left(\sigma^{-\frac{2}{q-2}}\right)^{k-2} \frac{t^{k-2}}{k!} \\
        & \leq \sigma^2 t^2 \sum_{k \geq 2} \left( \sigma^{-\frac{2}{q-2}} t\right)^{k-2}\frac{1}{k!}
\end{align*}
and since we have
\begin{align*}
    \sum_{k \geq 2} \left( \sigma^{-\frac{2}{q-2}} t\right)^{k-2}\frac{1}{k!}
    & = \frac12 + \sum_{k \geq 3} \left( \sigma^{-\frac{2}{q-2}} t\right)^{k-2}\frac{1}{k!} \\
    & = \frac12 + \sum_{k \geq 3} \left( \sigma^{-\frac{2}{q-2}} t\right)^{k-2}\frac{1}{2 \cdot 3 \cdot \ldots \cdot k} \\
    & \leq \frac12 + \sum_{k \geq 3}  \left( \sigma^{-\frac{2}{q-2}} t\right)^{k-2}\frac12 \cdot \frac{1}{3^{k-2}} \\
    & = \frac12 + \frac12 \sum_{k \geq 3} \left( \sigma^{-\frac{2}{q-2}} 
    \cdot \frac{t}{3}\right)^{k-2} \\
    & = \frac12 \sum_{k \geq 2} \left( \sigma^{-\frac{2}{q-2}} 
    \cdot \frac{t}{3} \right)^{k-2},
\end{align*}
we end up with the bound
\begin{equation*}
    \ell_0(t) + \ell_1(t) \leq \sigma^2 t^2\frac12 \sum_{k \geq 2} 
    \left( \sigma^{-\frac{2}{q-2}} \cdot \frac{t}{3} \right)^{k-2}.
\end{equation*}
Note that we arrived at a geometric series, so if 
$\sigma^{-\frac{2}{q-2}} \cdot \tfrac{t}{3} < 1$, we find that
\begin{equation*}
    \ell_0(t) + \ell_1(t) \leq \frac{\sigma^2 t^2}{2} 
    \cdot \frac{1}{1-\sigma^{-\frac{2}{q-2}} \cdot \frac{t}{3}}.
\end{equation*}
Based on this insight, we can now continue with
\begin{align*}
    \mathcal{T}[D^2(\ell_0+\ell_1)](\hat x) & = \inf_{t>0} \frac{D^2(\ell_0(t)+\ell_1(t))+\hat x}{t} \\
        & \leq \inf_{\substack{t>0\\\sigma^{-\frac{2}{q-2}} \cdot \tfrac{t}{3} < 1}} \frac{D^2(\ell_0(t)+\ell_1(t))+\hat x}{t} \\
        & \leq \inf_{\substack{t>0\\\sigma^{-\frac{2}{q-2}} \cdot \tfrac{t}{3} < 1}} \frac{D^2 \sigma^2 t^2}{2(1-\sigma^{-\frac{2}{q-2}} \cdot \frac{t}{3})t} + \frac{\hat x}{t} \\
        & = \inf_{\substack{t>0\\\sigma^{-\frac{2}{q-2}} \cdot \tfrac{t}{3}< 1}} \frac{D^2 \sigma^2 t}{2(1-\sigma^{-\frac{2}{q-2}} \cdot \frac{t}{3})} + \frac{\hat x}{t} \\
        & = \sigma^{-\frac{2}{q-2}} \cdot \frac{1}{\hat x} + \sqrt{2\hat x D^2 \sigma^2},
\end{align*}
where in the last step we used \cref{lem:bercuEtAl15_(2.17)} with 
$c=\sigma^{-\frac{2}{q-2}}/3$, $v=D^2 \sigma^2$, and $x=\hat x$.
\paragraph{Step 6: Final bound for $q>3$.}
Combining Steps 2, 4, 5, and 1, we have
\begin{align*}
     Q_{\norm{\tilde{M}_n}}^\infty(u) & \leq \min
     \left\{ 
            \mathcal{T}[D^2\ell_0](\hat x) +  \mathcal{T}[D^2(\ell_1 + \ell_2)](\hat x),
            \mathcal{T}[D^2(\ell_0 + \ell_1)](\hat x) +  \mathcal{T}[D^2\ell_2](\hat x)
            \right\} \\
    & \leq \min
    \left\{
            D\sigma\sqrt{2\hat x} + \alpha_{q,D} L 
            + D^2\frac{\hat x e}{3}\ell_1(q/e),
            \sigma^{-\frac{2}{q-2}}\frac{\hat x}{3} + D\sigma\sqrt{2\hat x} + \alpha_{q,D} L
        \right\} \\
    & = D\sigma\sqrt{2\hat x} + \alpha_{q,D} L + \frac{\hat x}{3}
    \min \left\{
            D^2 e\ell_1(q/e),  \sigma^{-\frac{2}{q-2}}
         \right\}.
\end{align*}
Let us turn to bounding $\min\{ e\ell_1(q/e),  \sigma^{-\frac{2}{q-2}}\}$, for which we will use a case distinction.
\textit{Case 1: $\sigma^{-\frac{2}{q-2}} \leq e$.}
In this case, immediately we have
\begin{align*}
    \min\{ D^2 e\ell_1(q/e),  \sigma^{-\frac{2}{q-2}}\} 
    \leq \min\{ e\ell_1(q/e), e\} \leq e \leq D^2e.
\end{align*}
\textit{Case 2: $\sigma^{-\frac{2}{q-2}} > e$.}
Note that this is equivalent to 
$\sigma^{\frac{2}{q-2}} < 1/ e$, so for all $k\in (2,q)$ 
we have
\begin{align*}
    \left(\sigma^{\frac{2}{q-2}}\right)^{q-k} = \sigma^{\frac{2(q-k)}{q-2}} \leq (1/e)^{q-k} = e^{k-q}.
\end{align*}
We can use this to show
\begin{align*}
     \ell_1(q/e) & = \sum_{ 2 < k < q} \sigma^{2 (q-k)/ (q-2)} \frac{(q/e)^k}{k!} \\
     & \leq e^{-q} \sum_{ 2 < k < q} \frac{q^k}{k!} \\
     & \leq e^{-q} \sum_{k \geq 0} \frac{q^k}{k!} \\
     & = e^{-q}e^q = 1,
\end{align*}
where we used in the first step that
\begin{align*}
    \sigma^{2 (q-k)/ (q-2)} (q/e)^k & =   \sigma^{2 (q-k)/ (q-2)} e^{-k} q^k \\
    & \leq e^{k-q}e^{-k} q^k \\
    & = e^{-q} q^k.
\end{align*}
We hence find that
\begin{align*}
     \min\{ D^2 e\ell_1(q/e),  \sigma^{-\frac{2}{q-2}}\} \leq \min\{ D^2e,  \sigma^{-\frac{2}{q-2}}\} \leq D^2 e.
\end{align*}
Altogether, we have
\begin{align*}
     Q_{\norm{\tilde{M}_n}}^\infty(u) & \leq 
         D\sigma\sqrt{2\hat x} + \alpha_{q,D} L + \frac{\hat x}{3}\min\{
            D^2 e\ell_1(q/e),  \sigma^{-\frac{2}{q-2}}
         \}  \\
        & \leq D\sigma\sqrt{2\hat x} + \alpha_{q,D} L + \frac{D^2 e\hat x}{3}.
\end{align*}
\paragraph{Step 7: Bound for all $q > 2$.}
Combining Step 3 (bound for $2 < q\leq 3$) and and Step 6 (bound for $ q >3$), 
and using the definitions
\begin{align*}
    \hat x & = \log(2/u), \\
    L & = e^{\hat x / q} = \exp(\log(2/u)/q) = (2/u)^{\frac{1}{q}},
\end{align*}
we finally get
\begin{align*}
    Q_{\norm{\tilde{M}_n}}^\infty(u) &  \leq D\sigma\sqrt{2\hat x} + \alpha_{q,D} L + \mathbbm{1}_{q > 3}\frac{e\hat x}{3} \\
    & = D \sigma\sqrt{2\log(2/u)} + \alpha_{q,D} (2/u)^{\frac{1}{q}} + \mathbbm{1}_{q > 3}\frac{D^2 e \log(2/u)}{3}
\end{align*}

Recall that we have
\begin{align*}
     Q^1_{\norm{M_n - \tilde{M}_n}}(u) & \leq  \frac{1}{u q L^{q-1}} \esssup \sum_{i=1}^n \E_{i-1} [ \norm{\xi_i}^q ], 
\end{align*}
so \eqref{eq:assumptions} and $L= (2/u)^{\frac{1}{q}}$ leads to
\begin{align*}
     Q^1_{\norm{M_n - \tilde{M}_n}}(u) & \leq  
      \frac{1}{u q \left((2/u)^{\frac{1}{q}}\right)^{q-1}} \\
      & = \frac{1}{u q (2/u)^{1-1/q}} \\
      & = u^{-1}u^{\frac{q-1}{q}} 2^{-\frac{q-1}{q}}
      = u^{-1/q}2^{1/q-1}.
\end{align*}
Using the bound on $Q_{\tilde{M}_n^*}(u)$ in \eqref{eq:quantile_function_bound}, we then obtain
\begin{align}
    Q_{M_n^*}(u) & \leq Q^1_{\norm{M_n - \tilde{M}_n}}(u) 
        + Q^\infty_{\norm{\tilde{M}_n}}(u) 
    \nonumber \\
    & \leq  \frac{1}{u q (2/u)^{1-1/q}} 
        +  D \sigma\sqrt{2\log(2/u)} 
        + \alpha_{q,D} (2/u)^{\frac{1}{q}} 
        + \mathbbm{1}_{q > 3}\frac{D^2 e \log(2/u)}{3} 
    \nonumber \\
    & =   D \sigma\sqrt{2\log(2/u)} + \left( \frac{1}{2q} 
        + \alpha_{q,D}\right) \left(\frac{2}{u}\right)^{1/q} 
        + \mathbbm{1}_{q > 3}\frac{D^2 e \log(2/u)}{3} 
    \nonumber \\
    & \leq   D \sigma\sqrt{2\log(2/u)} + \left( \frac{1}{2q} 
        + \alpha_{q,D} \right) \left(\frac{2}{u}\right)^{1/q} 
        +  \mathbbm{1}_{q > 3}\frac{D^2 q}{3}
            \left(\frac{2}{u}\right)^{1/q}
        & \text{(Lemma \ref{lem:log_poly_inequality})} 
    \nonumber \\
    & = D \sigma\sqrt{2\log(2/u)} + \left(\frac{1}{2q} 
        + \alpha_{q,D} + \mathbbm{1}_{q > 3}
            \frac{D^2 q}{3} \right)
            \left(\frac{2}{u}\right)^{1/q} 
    \nonumber \\
    & =  D \sigma\sqrt{2\log(2/u)} 
        + c_{q,D}\left(\frac{2}{u}\right)^{1/q}
    \label{eq:fn_bound}
\end{align}
with the constant
\begin{align*}
    c_{q,D} & := \frac{1}{2q} + \alpha_{q,D} + \mathbbm{1}_{q > 3}\frac{D^2 q}{3}  \\
        & = \frac{1}{2q} + \min\{1/q, 1/5\} + 1 + \mathbbm{1}_{q > 3}\frac{D^2 q}{3}.
\end{align*}

\subsection{Assembling the final bound}
Let $(M_i)_{0 \leq i \leq n}$ be a martingale in $\X$ adapted to
a nondecreasing filtration $(\sigalg_i)_{0 \leq i \leq n}$
with $M_0 = 0$,
and for the martingale differences $\xi_i := M_i - M_{i-1}$ we have
\begin{equation*}
    \sigma^2 :=
    \esssup \sum_{i=1}^n \E_{i-1} [ \norm{\xi_i}^2 ]
    < \infty
    \quad
    \text{ and }
    \quad
    C^q_q := 
    \esssup \sum_{i=1}^n \E_{i-1} [ \norm{\xi_i}^q ] < \infty
\end{equation*}
for some $q > 2$.
We can apply the bound \eqref{eq:fn_bound}
to the martingale $(M_i/C_q)_{0 \leq i \leq n}$,
which fulfills the original normalized assumptions
\eqref{eq:assumptions}
with $\sigma^2$ replaced by $\sigma^2/C_q^2$.
This leads to
\begin{equation*}
    \P \left[
        \max_{i \in \{1, \dots n \}} \norm{M_n/C_q}
        \leq D \sigma/C_q \sqrt{2\log(2/u)} + c_{q,D}\left(\frac{2}{u}\right)^{1/q}
    \right] 
    \geq 1-u
\end{equation*}
for all $u\in(0,1)$, so we finally obtain
\begin{equation*}
    \P \left[
        \max_{i \in \{1, \dots n \}} \norm{M_n}
        \leq D \sigma \sqrt{2\log(2/u)} + c_{q,D}C_q \left(\frac{2}{u}\right)^{1/q}
    \right] 
    \geq 1-u,
\end{equation*}
completing the proof.

\section*{Acknowledgements}
Christian Fiedler acknowledges funding from DFG Project FO 767/10-2 (eBer-24-32734) ``Implicit Bias in Adversarial Training''.

\bibliographystyle{abbrvnat}
\bibliography{references.bib}

\appendix

\section*{Appendices}

\section{Quantile functions}
\label{app:quantile_functions}

We collect general properties of quantile functions.
All results and proofs 
can be found in a more general form
in \citet{Pinelis2014}.\!\footnote{We note that the
results of \cite{Pinelis2014} are numbered differently
in the published version and preprint version. We refer to
the numbering of the published version.}
In this section, we exclusively consider real-valued random variables
defined on the same probability space.

\begin{lemma}[Submartingale inequality, \citealp{rio2017constants}, Lemma 2.3]
    \label{lem:submartingale_inequality}
    Let $(S_0, S_1, \dots, S_n)$ be a an integrable real-valued nonnegative 
    submartingale. Let $S_n^* = \max_{i=1, \dots,n} S_n$.
    Then we have 
    $Q_{S_n^*} \leq Q^1_{S_n}(u)$ for all $u \in (0, 1)$.
\end{lemma}

The integrated quantile functions allows for the following 
variational formulation originally due
to \citet{Rockafellar2000}.

\begin{lemma}[Variational formulation, \citealp{Pinelis2014}, Theorem 3.3]
    \label{lem:variational_formulation}
    We have
    \begin{equation*}
        Q_X^1(u) := 
        \inf_{t \in \R} \,
        t + \frac{ \E[ (X - t)_+ ] }{ u }.
    \end{equation*}
\end{lemma}

We collect three other general properties of $Q^1$ and $Q^\infty$.

\begin{lemma}[Quantile bounds,  \citealp{Pinelis2014}, Theorem 3.4]
    \label{lem:quantile_bounds}
    For all $u \in (0,1)$, we have
    \begin{equation*}
        Q_X(u) \leq Q^1_X(u) \leq Q^\infty_X(u).
    \end{equation*}
\end{lemma}

\begin{lemma}[Monotonicity of quantile functions,  \citealp{Pinelis2014}, Theorem 3.4]
    \label{lem:quantile_monotonicity}
    Let $X,Y$ be random variables with $X \leq Y$ almost surely,
    then we have for all $u \in (0,1)$ that
    \begin{equation*}
        Q_X(u) \leq Q_Y(u).
    \end{equation*}
\end{lemma}


\begin{lemma}[Subadditivity, \citealp{Pinelis2014}, Theorem 3.4]
    \label{lem:subadditivity}
    The functions $X \mapsto Q_X^1$ and $Y \mapsto Q_X^\infty$ 
    are subadditive
    in the sense that for all $X$ and $Y$, we have
    \begin{equation*}
        Q^1_{X + Y}(u)
        \leq
        Q^1_{X}(u) + Q^1_{Y}(u)
        \quad \text{and} \quad
        Q^1_{X + Y}(u)
        \leq
        Q^1_{X}(u) + Q^1_{Y}(u)
    \end{equation*}
    for all $u \in (0, 1)$.
\end{lemma}

The function $X \mapsto Q_X$ is generally \emph{not subadditive}.

\begin{remark}[Chernoff bound]
    \label{rem:quantile_chernoff_bound}
    The bound $Q_X(u) \leq Q_X^\infty(u)$
    contained in the statement of \cref{lem:quantile_bounds}
    captures the usual Chernoff bound
    performed to obtain a sharp tail bound
    in terms of the quantile function.
    In particular, we have    
    by Markov's inequality,
    for all $s\in\R$ and $t>0$ we have
    \begin{equation*}
        \P[X>s] \leq \E[\exp(tX)]\exp(-st).
    \end{equation*}
    Let now $u\in(0,1]$. For arbitrary, but fixed $t>0$ we can now define $s=s(t)=t^{-1}\ln(\E[\exp(tX)]/u)$, leading to
    \begin{equation*}
        \P[X>t^{-1}\ln(\E[\exp(tX)]/u)] \leq u,
    \end{equation*}
    and since $t>0$ was arbitrary and probability measures are continuous from above, we finally get
    \begin{equation*}
        \P[X> \inf_{t>0} t^{-1}\ln(\E[\exp(tX)]/u)] \leq u.
    \end{equation*}
    Using the definition of $Q_X^\infty$, we find that for all $u\in(0,1]$ it holds that
    \begin{equation*}
        \P[X > Q_X^\infty(u)] \leq u.
    \end{equation*}

\end{remark}

\section{Technical results}
\label{app:miscellaneous}

\begin{lemma}
\label{lem:log_poly_inequality}
For all $x > 0$ and $q > 0$, we have
    \begin{equation*}
        \log(x) \leq \frac{q}{e} x^{1/q}.
    \end{equation*}
\end{lemma}

\begin{proof}
    From the inequality $z \leq e^{z-1}$ valid for all $z \in \R$,
    we obtain $zq \leq \frac{q}{e} e^z$.
    For all $x > 0$, we may
    substitute $z = \log(x)/q$ into this inequality,
    proving the claim.
\end{proof}

The following claim is contained in \cite[Equation~(2.17)]{bercu2015concentration},
we provide a proof for completeness.

\begin{lemma} \label{lem:bercuEtAl15_(2.17)}
Let $c, x\in\Rp$ and $v\in\Rnn$. It holds that
\begin{equation}
    \inf_{\substack{t>0\\ ct < 1}} \frac{vt}{2(1-ct)} + \frac{x}{t} = cx + \sqrt{2xv}.
\end{equation}
\end{lemma}
\begin{proof}
Let $h(t) :=  \frac{vt}{2(1-ct)} + \frac{x}{t}$ for $t \in (0, 1/c)$,
then we have
\begin{equation*}
    h'(t)
    =
    \frac{v}{2(1-ct)^2}
    -
    \frac{x}{t^2}
\end{equation*}
and
\begin{equation*}
    h''(t) = \frac{cv}{(1-ct)^3} + \frac{2x}{t^3}.
\end{equation*}
Since $c,x\in\Rp$, we have $h''(t)>0$ for $t\in(0,1/c)$, so $h$ is convex.
We determine the root $t^*$ of $h'$ as
\begin{equation*}
    t^* 
    =
    \frac{\sqrt{2x}}{c\sqrt{2x} + \sqrt{v}}
    \in (0, 1/c).
\end{equation*}
After elementary calculations, we see that $h(t^*) =cx + \sqrt{2xv}$,
proving the claim.
\end{proof}

\end{document}